\documentclass[12pt,a4paper]{amsart}

\usepackage[utf8]{inputenc}
\usepackage{orcidlink}
\usepackage{amssymb}
\usepackage{enumerate}
\usepackage{amsmath}
\usepackage{xcolor,color,graphics}
\usepackage{mathtools}
\usepackage{algpseudocode}
\usepackage{hyperref}
\hypersetup{colorlinks=true,allcolors=blue}
\usepackage[margin=2.5cm]{geometry}

\newtheorem{proposition}{Proposition}
\newtheorem{lemma}[proposition]{Lemma}
\newtheorem{corollary}[proposition]{Corollary}
\newtheorem{theorem}[proposition]{Theorem}

\theoremstyle{definition}
\newtheorem{remark}[proposition]{Remark}
\newtheorem{example}[proposition]{Example}

\newcommand{\R}{\mathbb{R}}
\newcommand{\I}{\mathbb{I}}

\title{Quasi-copulas as linear combinations of copulas}

\author[G. Dolinar]{Gregor Dolinar \orcidlink{0000-0001-9083-5578} }
\address{University of Ljubljana, Faculty of Electrical Engineering, and Institute of Mathematics, Physics and Mechanics, Ljubljana, Slovenia}
\email{gregor.dolinar@fe.uni-lj.si}

\author[B. Kuzma]{Bojan Kuzma \orcidlink{0000-0003-3478-060X} } 
\address{University of Primorska, Faculty of Mathematics, Natural Sciences and Information Technologies, Koper, Slovenia and Institute of Mathematics, Physics and Mechanics, Ljubljana, Slovenia}
\email{bojan.kuzma@upr.si}

\author[N. Stopar]{Nik Stopar \orcidlink{0000-0002-0004-4957} } 
\address{University of Ljubljana, Faculty of Civil and Geodetic Engineering, and Institute of Mathematics, Physics and Mechanics, Ljubljana, Slovenia}
\email{nik.stopar@fgg.uni-lj.si}

\keywords{Quasi-copula; copula; linear combination; affine combination; Minkowski norm}

\subjclass[2020]{62H05, 60E05}

\begin{document}

\begin{abstract}
We prove  that every quasi-copula can be written as a uniformly converging infinite sum of multiples of copulas. 
Furthermore, we characterize those quasi-copulas which can be written as a finite sum of multiples of copulas, i.e., that are a linear combination of two copulas. This generalizes a recent result of 
Fern\'{a}ndez-S\'{a}nchez, Quesada-Molina, and \'{U}beda-Flores who considered linear combinations of discrete copulas. 
\end{abstract}

\maketitle

\section{Introduction}

Since their introduction by Sklar in 1959, copulas have become an important  tool in statistical literature, because they describe all possible dependencies between random variables.
Their widespread use has stimulated investigations of their structural properties and has lead to several important generalizations.

One of the generalizations most closely related to copulas, and perhaps the most fruitful one, are quasi-copulas.
They were introduced in 1993 by Alsina, Nelsen and Scheizer \cite{AlsNelSch93} in order to characterize certain operations on distribution functions. Quasi-copulas and their connections to copulas have since been intensively investigated in order to better understand the set of copulas, see \cite{AriMesDeB17,DurFerTru16,KikLab14,GenQueRodSem99,RodUbe09}.
Quasi-copulas appear naturally in dependence modeling when studying lower and upper bounds of sets of copulas \cite{NelQueRodUbe04,KleKokOmlSamSto22,MarSadShi10}, because point-wise infima and suprema of sets of copulas are always quasi-copulas.
In particular, they are essential in the setting of imprecise probabilities modeled by probability boxes \cite{PelVicMonMir13,MonMirPelVic15,OmlSto20,Sto23}.
They have also become popular in the theory of aggregation functions and in fuzzy set theory \cite{DeBDeMMes12,Kol03,FodDeB08}.  
For an overview of results on quasi-copulas and some recent developments we refer the reader to a survey paper \cite{AriMesDeB20}.

It is well known that the set $\mathcal{C}$ of all bivariate copulas is a convex set. Going beyond convex combinations, the authors in \cite{DarOls95} considered linear combinations and introduced the linear span of all bivariate copulas, which we denote by $\mathcal{S}=\operatorname{span}\mathcal{C}$.
Due to convexity of $\mathcal{C}$, every element of $\mathcal{S}$ can be written as a linear (in fact, affine) combination of two copulas.
The vector space $\mathcal{S}$ can be equipped with the so called \emph{Minkowski norm}, defined by
\begin{equation}\label{eq:Minkowski}
\|A\|_M=\inf\{s+t \mid s,t \ge 0,\ A=sB-tC \text{ for some } B,C \in \mathcal{C}\}
\end{equation}
for all $A \in \mathcal{S}$.
It was shown in \cite{DarOls95} that $\mathcal{S}$ equipped with this norm is a Banach space.
Furthermore, if we additionally equip $\mathcal{S}$ with the product of copulas, also called \emph{Markov product}, introduced in \cite{DarNguOls92},
then $\mathcal{S}$ becomes a Banach algebra.
This Banach algebra can be used to study one-parametric semigroups of copulas.
For further information on this topic we refer the reader to \cite{DurSem16,DarOls95} and the references therein.

It was recently shown in \cite{FerQueUbe21} that every discrete quasi-copula is a linear combination of discrete copulas. This means that in the discrete setting $\mathcal{S}$ contains all discrete quasi-copulas.
In contrast, in the continuous setting $\mathcal{S}$ does not contain all quasi-copulas, since any quasi-copula in $\mathcal{S}$ induces a signed measure (on the Borel $\sigma$-algebra in $\I^2$), but there are quasi-copulas that do not induce a signed measure \cite{FerRodUbe11}.

The main goal of  the present paper is to initiate the study of quasi-copulas as convergent series of scalar multiples of copulas. We extend the aforementioned result of \cite{FerQueUbe21} to general quasi-copulas by showing that any bivariate quasi-copula $Q \colon \I^2 \to \I$ is an infinite linear combination of bivariate copulas, i.e., it can be expressed as an infinite sum of multiples of copulas, where the sum converges uniformly. Moreover, all the partial sums, though not necessarily quasi-copulas, are nevertheless positive multiples of quasi-copulas.
As a consequence we show that the closure of $\mathcal{S}$ in the uniform norm contains all quasi-copulas.

Furthermore, we characterize quasi-copulas that lie in $\mathcal{S}$, i.e., can be expressed as an affine combination of two copulas.
We do this by closely examining how the mass of a (discrete) $2$-increasing function can be dominated with the mass of a multiple of a (discrete) copula.
Our results shed some new light on an open question posed in \cite{AriMesDeB20} on quasi-copulas that induce a signed measure,
because every quasi-copula in $\mathcal{S}$ indeed induces a signed measure.
However, we show by an example that there are quasi-copulas that induce a signed measure but do not lie in $\mathcal{S}$.

The paper is structured as follows. In Section~\ref{sec:pre} we give the necessary definitions and fix some notations.
Section~\ref{sec:discrete} is devoted to discrete quasi-copulas. We investigate decompositions of discrete quasi-copulas as linear combinations of discrete copulas, and the domination of mass of a discrete function with the mass of a discrete copula.
In Section~\ref{sec:general} we present our main results on finite and infinite linear combinations of general copulas along with a counterexample showing that not every quasi-copula which induces a signed measure can be written as a linear combination of two copulas.

\section{Preliminaries}\label{sec:pre}

We will denote by $\I=[0,1]$ the unit interval and by $\I^2=\I \times \I$ the unit square. Recall that a \emph{quasi-copula} is a function $Q \colon \I^2 \to \I$ that satisfies conditions
\begin{enumerate}[$(i)$]
    \item $Q$ is \emph{grounded}, i.e., $Q(x,0)=Q(0,y)=0$ for all $x,y \in \I$,
    \item $Q$ has  \emph{uniform marginals}, i.e., $Q(x,1)=x$ and $Q(1,y)=y$ for all $x,y \in \I$,
    \item $Q$ is \emph{increasing} in each of its arguments,
    \item $Q$ is \emph{$1$-Lipschitz}, i.e.,
    $$|Q(x_2,y_2)-Q(x_1,y_1)| \leq |x_2-x_1|+|y_2-y_1|$$
    for all $(x_1,y_1), (x_2,y_2) \in \I^2$.
\end{enumerate}
A \emph{copula} is a function $Q \colon \I^2 \to \I$ that satisfies conditions $(i)$, $(ii)$ and
\begin{enumerate}[$(i)$]\setcounter{enumi}{4}
    \item $Q$ is \emph{$2$-increasing}, i.e.,    $$Q(x_2,y_2)+Q(x_1,y_1)-Q(x_2,y_1)-Q(x_1,y_2) \geq 0$$
    for all $(x_1,y_1), (x_2,y_2) \in \I^2$ with $x_1\leq x_2$ and $y_1\leq y_2$.
\end{enumerate}
The conditions $(i)$, $(ii)$ and $(v)$ together imply conditions $(iii)$ and $(iv)$, hence every copula is a quasi-copula. A quasi-copula that is not a copula is usually called a \emph{proper} quasi-copula.

If in the above definitions we replace the domain $\I^2$ of the function $Q$ by a finite mesh 
\begin{equation}\label{eq:mesh}
\begin{aligned}
    \delta_1 \times \delta_2=\{x_1,x_2,\ldots,x_{n+1}\} &\times \{y_1,y_2,\ldots,y_{m+1}\} \quad\hbox{ with} \\[1mm]
 0=x_1<x_2<\ldots<x_{n+1}=1&, \qquad 0=y_1<y_2<\ldots<y_{m+1}=1, 
 \end{aligned}
\end{equation}
and restrict the conditions to the mesh, we obtain the definition of a \emph{discrete quasi-copula} and \emph{discrete copula}, respectively.
Note that this definition of a discrete (quasi) copula is slightly more general than those given in \cite{DurSem16,QueSem05} (cf. \cite{OmlSto20}).

Given a rectangle $R=[x_1,x_2] \times [y_1,y_2]$ and a (discrete) quasi-copula $Q$ whose domain contains the vertices of $R$, the $Q$-volume of $R$ is defined by
$$V_Q(R)=Q(x_2,y_2)-Q(x_1,y_2)-Q(x_2,y_1)+Q(x_1,y_1).$$
Note that the volume of rectangle $R$ with respect to the product copula $\Pi(x,y)=xy$ is just the Lebesgue measure of $R$, which will be denoted by
$$\lambda^2(R)=(x_2-x_1)(y_2-y_1).$$

We will denote by $\mathcal{C}$ the set of all bivariate copulas and by $\mathcal{S}=\operatorname{span}\mathcal{C}$ their linear span.
The set $\mathcal{S}$ can be equipped with the Minkowski norm defined in \eqref{eq:Minkowski}.
By \cite[Lemma~3.2]{DarOls95} the infimum in the definition of Minkowski norm is actually a minimum.
If we let $\mathcal{B}=\operatorname{co} (-\mathcal{C} \cup \mathcal{C})$ denote the convex hull of the set $-\mathcal{C} \cup \mathcal{C}$, then $\mathcal{S}=\cup_{t \ge 0} t\mathcal{B}$ and
$$\|A\|_M=\inf\{t>0 \mid A \in t\mathcal{B}\}$$
for all $A \in \mathcal{S}$.
We can similarly introduce the sets $\mathcal{C}$ and $\mathcal{S}$, and the norm $\|.\|_M$ in the context of discrete functions on a fixed finite mesh.
We will maintain the same notations for clarity, since it should be clear from the context which case we are considering.

\section{Mass domination and discrete quasi-copulas}\label{sec:discrete}

In this section we investigate decompositions of discrete quasi-copulas as linear combinations of discrete copulas. Throughout the section we fix a finite mesh $\delta_1\times\delta_2$ as defined in~\eqref{eq:mesh}.
We denote the minimal rectangles of the mesh $\delta_1\times\delta_2$ by
$$R_{ij}=[x_i,x_{i+1}]\times[y_j,y_{j+1}], \qquad 1 \leq i \leq n,\ 1 \leq j \leq m.$$
Note that a discrete quasi-copula $Q$ defined on $\delta_1\times\delta_2$ is a discrete copula if and only if $V_Q(R_{ij}) \geq 0$ for all $i \in \{1,2,\ldots,n\}$ and $j \in \{1,2,\ldots,m\}$.

In \cite[Corollary~11]{FerQueUbe21} it was shown using topological methods that every discrete quasi-copula (defined on an equidistant mesh) can be written as a linear combination of discrete copulas.
The sum of coefficients is automatically equal to $1$, so the linear combination is actually affine.
Since convex combinations of copulas are copulas, this easily implies that every discrete quasi-copula $Q$ can be written as a linear combination of two discrete copulas (in fact, as a difference of two nonnegative multiples of copulas). 
We give here a direct proof of this fact, which even shows that one of the two copulas may be the product copula $\Pi$.
The idea behind this proof is to dominate the mass of $Q$ with a multiple of the mass of $\Pi$. The concept of mass domination will be crucial in the rest of the paper.

\begin{lemma}[{\cite[Corollary~11]{FerQueUbe21}}]\label{lem:1}
For every discrete
quasi-copula $Q$ defined on a finite mesh $\delta_1\times\delta_2$ there exist discrete copulas $C_1$ and $C_2$ defined on $\delta_1\times\delta_2$ and real numbers $\alpha_1 \geq 1$ and $\alpha_2\leq 0$ with $\alpha_1+\alpha_2=1$ such that $Q(x,y)=\alpha_1 C_1(x,y)+\alpha_2 C_2(x,y)$ for all $(x,y)\in \delta_1\times\delta_2$.
\end{lemma}

\begin{proof}
Define
$$\alpha_1:=\max_{\substack{i=1,\ldots,n \\ j=1,\ldots,m}}\frac{V_Q(R_{ij})}{\lambda^2(R_{ij})}$$
and let  $C_1(x_i,y_j)=x_iy_j$ be the product copula restricted to $\delta_1\times \delta_2$.
Note that
\begin{equation}\label{eq:sum}
    \sum_{i=1}^n \sum_{j=1}^m V_Q(R_{ij})=1=\sum_{i=1}^n \sum_{j=1}^m \lambda^2(R_{ij}),
\end{equation}
so that $V_Q(R_{ij})\ge \lambda^2(R_{ij})$ for at least one rectangle $R_{ij}$, and therefore, $\alpha_1\ge 1$.
If $\alpha_1=1$ then $V_Q(R_{ij})\le \lambda^2(R_{ij})$ for all $i \in \{1,2,\ldots,n\}$ and $j \in \{1,2,\ldots,m\}$ and equation~\eqref{eq:sum} implies $V_Q(R_{ij})= \lambda^2(R_{ij})$ for all $i \in \{1,2,\ldots,n\}$ and $j \in \{1,2,\ldots,m\}$.
In this case, $Q=C_1$ is the   product copula restricted to $\delta_1\times \delta_2$ and we take $\alpha_2=0$. 
Otherwise let $\alpha_2=1-\alpha_1<0$
  and define $C_2=\frac{1}{
\alpha_2}(Q-\alpha_1 C_1)$. Let us show that  $C_2$ is a copula. Clearly $C_2$ is grounded and has uniform marginals. To verify that $C_2$ is $2$-increasing it suffices to show that $V_{C_2}(R_{ij})\ge0$ for all $i \in \{1,2,\ldots,n\}$ and $j \in \{1,2,\ldots,m\}$. Indeed, 
$$V_{C_2}(R_{ij})= \tfrac{1}{\alpha_2}\bigl(V_{Q}(R_{ij})-\alpha_1 V_{C_1}(R_{ij})\bigr)=
\tfrac{1}{\alpha_2}\Bigl(V_Q(R_{ij})-\alpha_1 \lambda^2(R_{ij})\Bigr)\ge0$$
because $\alpha_2<0$ and $V_Q(R_{ij})-\alpha_1 \lambda^2(R_{ij})\le0$ by the definition of $\alpha_1$.
\end{proof}
\begin{remark}
    Notice that in the
proof of Lemma~\ref{lem:1}, the copula $C_1$ can be any discrete copula which  satisfies
$V_{C_1}(R_{ij}) > 0$ for all $i,j$. Namely, it suffices  to set $\alpha_1=\max_{\substack{i=1,\ldots,n\\ j=1,\ldots,m}}\frac{V_Q(R_{ij})}{V_{C_1}(R_{ij})}$.
\end{remark}
Next, we investigate more closely how the mass of a grounded $2$-increasing function $A\colon\delta_1\times\delta_2\to\R$ can be dominated by a multiple of the mass of a discrete copula.
For a nonzero grounded $2$-increasing function $A\colon\delta_1\times\delta_2\to\R$ define
\begin{equation}\label{eq:alpha}
   \alpha_A =\max_{\substack{i=1,\ldots,n\\ j=1,\ldots,m}} \left\{
   \frac{V_A([x_{i},x_{i+1}]\times\I)}{x_{i+1}-x_i}\,,\,
    \frac{V_A(\I \times [y_{j},y_{j+1}])}{y_{j+1}-y_{j}}
   \right\}
\end{equation}
and for each $(x,y)\in\delta_1\times\delta_2$ let
\begin{equation}\label{eq:underlineC}
\underline{C}_A(x,y)=\max\left\{\frac{V_A([0,x]\times[0,y])}{\alpha_A}, x+y-1+\frac{V_A([x,1]\times[y,1])}{\alpha_A}\right\}
\end{equation}
and 
\begin{equation}\label{eq:overlineC}
\overline{C}_A(x,y)=\min\left\{x-\frac{V_A([0,x]\times[y,1])}{\alpha_A}, y-\frac{V_A([x,1]\times[0,y])}{\alpha_A}\right\}.
\end{equation}
We remark that the conditions for $A$ imply $\alpha_A \neq 0$. On the other hand, if $A$ is a zero function, then $\alpha_A=0$.

Our first result shows that 
$\underline{C}_A$ and  $\overline{C}_A$ are discrete copulas.

\begin{proposition}\label{prop:2}
For every nonzero grounded $2$-increasing function $A \colon \delta_1 \times \delta_2 \to \R$
the functions $\underline{C}_A$ and  $\overline{C}_A$ are discrete copulas. Moreover,
$$\alpha_A V_{\underline{C}_A}(R_{ij})\ge V_A(R_{ij})  \quad\text{and}\quad \alpha_A V_{\overline{C}_A}(R_{ij})\ge V_A(R_{ij})$$
for all $i \in \{1,2,\ldots,n\}$ and $j \in \{1,2,\ldots,m\}$.
\end{proposition}

\begin{proof}
By the definition of $\alpha_A $ we have 
\begin{equation}\label{eq:V_A}
  V_A(\I\times [y_j,y_{j+1}])\le \alpha_A({y_{j+1}-y_j}).  
\end{equation}
Summing up these  inequalities over all $j\in\{1,\dots,k-1\}$, and dividing by $\alpha_A$, we obtain 
 $$\frac{V_A(\I\times [0,y_k])}{\alpha_A}\le y_k.$$
Therefore, $\overline{C}_A(0,y_k)=\min\{0-\tfrac{0}{\alpha_A}, y_k-\frac{V_A(\I\times [0,y_k])}{\alpha_A}\}=0$ for every $k$. Similarly we see that $\overline{C}_A(x_l,0)=0$ for every $l$, so that $\overline{C}_A$ is grounded. 
Moreover, we also find that $\overline{C}_A(1,y_k)=\min\{\frac{V_A(\I\times[0,y_k])}{\alpha_A},y_k-0\}=y_k$. Similarly we see that $\overline{C}_A(x_l,1)=x_l$, i.e., $\overline{C}_A$ has uniform marginals.   

Summing up the  inequalities in \eqref{eq:V_A} over all $j\in\{k,\dots,m\}$, and dividing by $\alpha_A$, we obtain $$\frac{V_A(\I\times [y_k,1])}{\alpha_A}\le(1-y_k).$$
Therefore, $\underline{C}_A(0,y_k)=\max\{0, y_k-1+\frac{V_A(\I\times [y_k,1])}{\alpha_A}\}=0$ for every $k$.
Similarly we see that $\underline{C}_A(x_l,0)=0$ for every $l$, so that $\underline{C}_A$ is grounded.
 Likewise, if in \eqref{eq:V_A} we sum up over all indices $j\in\{1,\dots,k\}$ we see that  $\underline{C}_A(1,y_k)=\max\{\frac{V_A(\I\times[0,y_k])}{\alpha_A},y_k\}=y_k$. Similarly we see that $\underline{C}_A(x_l,1)=x_l$, i.e., $\underline{C}_A$ has uniform marginals.
 
It remains to show that $\underline{C}_A$ and $\overline{C}_A$ are $2$-increasing.
Let $R_{ij}=[x_i,x_{i+1}]\times[y_j,y_{j+1}]$ be a rectangle with vertices in the mesh and denote
\begin{align*}
    a_{kt}&:=\alpha_A x_{k}-V_A([0,x_{k}]\times[y_{t},1]),\\
    b_{kt}&:=\alpha_Ay_{t}-V_A([x_{k},1]\times[0,y_{t}]),
\end{align*}
to simplify the notations. 
Using the identity
\begin{equation}\label{eq:min-max}
 \begin{aligned}
 \min\{\alpha,\beta\}&-\min\{\gamma,\delta\}=\min\Bigl\{\alpha-\min\{\gamma,\delta\},\beta-\min\{\gamma,\delta\}\Bigr\}\\
 &=\min\Bigl\{\max\{\alpha-\gamma,\alpha-\delta\}\,,\,\max\{\beta-\gamma,\beta-\delta\}\Bigr\} 
 \end{aligned}
\end{equation}
we then compute
\begin{align}
\alpha_A V_{\overline{C}_A}(R_{ij})&=\alpha_A\Bigl(\bigl(\overline{C}_A(x_{i+1},y_{j+1})-\overline{C}_A(x_{i+1},y_{j})\bigr)-\bigl( \overline{C}_A(x_{i},y_{j+1})-\overline{C}_A(x_{i},y_{j})\bigr)\Bigr)\nonumber\\
&=\Bigl(\min\{a_{(i+1)(j+1)},b_{(i+1)(j+1)}\}-\min\{a_{(i+1)j},b_{(i+1)j}\}\Bigr)\nonumber\\
&\hspace{4cm}\mbox{}-\Bigl(\min\{a_{i(j+1)},b_{i(j+1)}\}-\min\{a_{ij},b_{ij}\}\Bigr)\nonumber\\
&=\min\Bigl\{\max\{\zeta_1,\zeta_2\},\max\{\zeta_3,\zeta_4\}\Bigr\}\label{eq:min(i+1)}\\
&\hspace{5cm}\mbox{}-\min\Bigl\{\max\{\zeta'_1,\zeta'_2\},\max\{\zeta'_3,\zeta'_4\}\Bigr\} \nonumber
\end{align} 
where in  \eqref{eq:min(i+1)} we used \eqref{eq:min-max} and introduced
\begin{align*}
\zeta_1 &:=a_{(i+1)(j+1)}-a_{(i+1)j}, &\zeta_2 &:=a_{(i+1)(j+1)}-b_{(i+1)j},\\
\zeta_3 &:=b_{(i+1)(j+1)}-a_{(i+1)j}, &\zeta_4 &:=b_{(i+1)(j+1)}-b_{(i+1)j},
\end{align*}
and where  $\zeta'_k$ denote the same quantities as $\zeta_k$ except with $i+1$ replaced by $i$, that is,
\begin{align*}
&&\zeta'_1 &:=a_{i(j+1)}-a_{ij}, &\zeta'_2 &:=a_{i(j+1)}-b_{ij},&&\\
&&\zeta'_3 &:=b_{i(j+1)}-a_{ij}, &\zeta'_4 &:=b_{i(j+1)}-b_{ij}.&&
\end{align*}
It follows that
\begin{align*}
    \zeta_1&=V_A([0,x_{i+1}]\times[y_j,1])-V_A([0,x_{i+1}]\times[y_{j+1},1])=V_A([0,x_{i+1}]\times[y_j,y_{j+1}]),\\
    \zeta_4&=\alpha_A(y_{j+1}-y_j)-\bigl(V_A([x_{i+1},1]\times[0,y_{j+1}])-V_A([x_{i+1},1]\times[0,y_{j}])\bigr)\\
    &=\alpha_A(y_{j+1}-y_j)-V_A([x_{i+1},1]\times[y_j,y_{j+1}]).
\end{align*}
Thus, by definition of $\alpha_A$,
$$\zeta_4-\zeta_1=\alpha_A(y_{j+1}-y_j)-V_A(\I\times[y_j,y_{j+1}])\ge0,$$
or equivalently, $\zeta_4\ge\zeta_1$.
Furthermore,
\begin{align*}
    \zeta_2&=\alpha_A x_{i+1}-V_A([0,x_{i+1}]\times[y_{j+1},1])-\alpha_A y_j+V_A([x_{i+1},1]\times[0,y_j])\\
    &=\alpha_A(x_{i+1}-y_j)+V_A([x_{i+1},1]\times[0,y_j])-V_A([0,x_{i+1}]\times[y_{j+1},1]).
\end{align*}
Clearly, $\zeta_1+\zeta_4=\zeta_2+\zeta_3$, that is, the arithmetic mean of $\zeta_1$ and $\zeta_4$ is the same as the arithmetic mean of $\zeta_2$ and $\zeta_3$.
Using also $\zeta_1\le \zeta_4$ one can check that there are only four possibilities:
$$\zeta_1\le \zeta_2\le\zeta_3\le\zeta_4,\quad \zeta_1\le \zeta_3\le\zeta_2\le\zeta_4,\quad \zeta_2\le \zeta_1\le\zeta_4\le\zeta_3 ,\quad \hbox{and}\quad\zeta_3\le \zeta_1\le\zeta_4\le\zeta_2.$$ In each case, one easily sees that
$$\min\Bigl\{\max\{\zeta_1,\zeta_2\},\max\{\zeta_3,\zeta_4\}\Bigr\}=\min\Bigl\{\max\{\zeta_1,\zeta_2\},\zeta_4\Bigr\}.$$
In the same way we see that a similar equality holds for $\zeta'_1,\ldots,\zeta'_4$.
Therefore, we obtain
\begin{equation}\label{eq:min-max-a1a2a3}
    \alpha_A V_{\overline{C}_A}(R_{ij})=\min\Bigl\{\max\{\zeta_1,\zeta_2\},\zeta_4\Bigr\}-\min\Bigl\{\max\{\zeta_1',\zeta_2'\},\zeta_4'\Bigr\}.
\end{equation}
Note that
\begin{align*}
    \zeta_1-\zeta_1'&=V_A([0,x_{i+1}]\times[y_j,y_{j+1}])-V_A([0,x_{i}]\times[y_j,y_{j+1}])=V_A(R_{ij}),\\
   \zeta_4-\zeta_4'&=\alpha_A(y_{j+1}-y_j)-V_A([x_{i+1},1]\times[y_j,y_{j+1}])\\
   &\hspace{1cm}\mbox{}-\alpha_A(y_{j+1}-y_j)+V_A([x_{i},1]\times[y_j,y_{j+1}])=V_A(R_{ij}),  \\
   \zeta_2-\zeta_2'&=\Bigl(\alpha_A(x_{i+1}-y_j)+V_A([x_{i+1},1]\times[0,y_j])-V_A([0,x_{i+1}]\times[y_{j+1},1])\Bigr)\\
   &\hspace{1cm}\mbox{}-\Bigl(\alpha_A(x_{i}-y_j)+V_A([x_{i},1]\times[0,y_j])-V_A([0,x_{i}]\times[y_{j+1},1]) \Bigr)\\
   &=\alpha_A(x_{i+1}-x_i)-V_A([x_i,x_{i+1}]\times[0,y_j])-V_A([x_i,x_{i+1}]\times[y_{j+1},1])\\
   &=\alpha_A(x_{i+1}-x_i)- V_A([x_i,x_{i+1}]\times\I)+V_A(R_{ij})=V_A(R_{ij})+h,
\end{align*}
where $h=\alpha_A(x_{i+1}-x_i)- V_A([x_i,x_{i+1}]\times\I) \geq 0$ by the definition of $\alpha_A$.
Using these three equalities in \eqref{eq:min-max-a1a2a3} we get
\begin{align*}
 \alpha_A V_{\overline{C}_A}(R_{ij})&=\min\Bigl\{\max\{V_A(R_{ij})+\zeta_1'\,,\,V_A(R_{ij})+h+\zeta_2'\}\,,\,V_A(R_{ij})+\zeta_4'\Bigr\}\\
 &\hspace{1cm}\mbox{}-\min\Bigl\{\max\{\zeta_1',\zeta_2'\},\zeta_4'\Bigr\}\\
 &\ge V_A(R_{ij}),   
\end{align*}
where the last inequality  follows because $h\ge0$ and $\min$ and $\max$ are both increasing functions in any of their arguments.
Since $A$ is $2$-increasing we have that $V_A(R_{ij})\ge0$, so $\overline{C}_A$ is a discrete copula.

Consider now $\underline{C}_A$. Similarly  as above we denote 
\begin{align*}
    \hat{a}_{kt}&:=V_A([0,x_{k}]\times[0,y_{t}]),\\
    \hat{b}_{kt}&:=\alpha_A(x_k+y_t-1)+V_A([x_{k},1]\times[y_{t},1]),
\end{align*}
and 
\begin{align*}
\hat{\zeta}_1&:=\hat{a}_{(i+1)(j+1)}-\hat{a}_{(i+1)j}=V_A([0,x_{i+1}]\times[y_j,y_{j+1}]),\\
\hat{\zeta}_2&:=\hat{a}_{(i+1)(j+1)}-\hat{b}_{(i+1)j},
\\
\hat{\zeta}_3&:=\hat{b}_{(i+1)(j+1)}-\hat{a}_{(i+1)j},
\\
\hat{\zeta}_4&:=\hat{b}_{(i+1)(j+1)}-\hat{b}_{(i+1)j}=\alpha_A(y_{j+1}-y_j)-V_A([x_{i+1},1]\times[y_{j},y_{j+1}]),
\end{align*}
and we let $\hat{\zeta}'_k$ denote the same quantities as $\hat{\zeta}_k$ but with $i+1$ replaced by~$i$.
By using the fact that
$$\max\{\alpha,\beta\}-\max\{\gamma,\delta\}=\max\Bigl\{\min\{\alpha-\gamma,\alpha-\delta\},\min\{\beta-\gamma,\beta-\delta\}\Bigr\} $$ 
one computes 
\begin{align*}
\alpha_A V_{\underline{C}_A}(R_{ij})&=\alpha_A\Bigl(\bigl(\underline{C}_A(x_{i+1},y_{j+1})-\underline{C}_A(x_{i+1},y_{j})\bigl)-\bigl( \underline{C}_A(x_{i},y_{j+1})-\underline{C}_A(x_{i},y_{j})\bigr)\Bigr)\nonumber\\
&=\max\Bigl\{\min\{\hat{\zeta}_1,\hat{\zeta}_2\},\min\{\hat{\zeta}_3,\hat{\zeta}_4\}\Bigr\} \label{eq:min(i+1)}-\max\Bigl\{\min\{\hat{\zeta}'_1,\hat{\zeta}'_2\},\min\{\hat{\zeta}'_3,\hat{\zeta}'_4\}\Bigr\}.
\end{align*}
Note that
\begin{align*}
\hat{\zeta}_4-\hat{\zeta}_1&=\hat{\zeta}'_4-\hat{\zeta}'_1=\alpha_A(y_{j+1}-y_j)-V_A(\I\times[y_j,y_{j+1}])\ge0,\\
\hat{\zeta}_1+\hat{\zeta}_4&=\hat{\zeta}_2+\hat{\zeta}_3,\\ \hat{\zeta}'_1+\hat{\zeta}'_4&=\hat{\zeta}'_2+\hat{\zeta}'_3,\\
\hat{\zeta}_1-\hat{\zeta}'_1&=V_A(R_{ij}),\\
\hat{\zeta}_4-\hat{\zeta}'_4&=V_A(R_{ij}),\\
\hat{\zeta}_3-\hat{\zeta}'_3&=V_A(R_{ij}) +\alpha_A(x_{i+1}-x_i)-V_A([x_i,x_{i+1}]\times\I)=V_A(R_{ij})+\hat{h},
\end{align*}
where $\hat{h}=\alpha_A(x_{i+1}-x_i)-V_A([x_i,x_{i+1}]\times\I)\ge0$. From the first three equations we deduce
$$\max\Bigl\{\min\{\hat{\zeta}_1,\hat{\zeta}_2\}\,,\,\min\{\hat{\zeta}_3,\hat{\zeta}_4\}\Bigr\}=\max\Bigl\{\hat{\zeta}_1,\min\{\hat{\zeta}_3,\hat{\zeta}_4\}\Bigr\}$$
and the same holds for $\hat{\zeta}'_1,\ldots,\hat{\zeta}'_4$.
Therefore,
\begin{align*}
\alpha_A V_{\underline{C}_A}(R_{ij})&=\max\Bigl\{\hat{\zeta}'_1+V_A(R_{ij}),\min\{\hat{\zeta}'_3+V_A(R_{ij})+\hat{h},\hat{\zeta}'_4+V_A(R_{ij}\}\Bigr\}\\
&\hspace{1cm}\mbox{}-\max\Bigl\{\hat{\zeta}'_1,\min\{\hat{\zeta}'_3,\hat{\zeta}'_4\}\Bigr\}\\
&\ge V_A(R_{ij}) \ge 0,
\end{align*}
hence, $\underline{C}_A$ is a discrete copula.
\end{proof}

We illustrate Proposition~\ref{prop:2} with an example.

\begin{example}
Let $\delta_1=\delta_2=\{0,\frac14,\frac24,\frac34,1\}$ and let $A \colon \delta_1 \times \delta_2 \to \R$ be a function given by the matrix of its values
$$A=(A(x_i,y_j))_{ij}=\begin{pmatrix}
    0 & 3 & 8 & 16 & 21 \\
    0 & 1 & 5 & 11 & 15 \\
    0 & 1 & 4 & 10 & 11 \\
    0 & 1 & 1 &  3 &  4 \\
    0 & 0 & 0 &  0 &  0
\end{pmatrix}.$$
Here and in the rest of the paper it should be understood that index $i$ in the matrix runs from left to right and index $j$ runs from bottom to top, so that the bottom left corner of the matrix corresponds to the origin of the coordinate system in the domain of $A$, index $i$~corresponds to coordinate $x$, and index $j$ to coordinate $y$.
The distribution of mass of function $A$ is given by the matrix
$$S=(V_A(R_{ij}))_{ij}=\begin{pmatrix}
    2 & 1 & 2 & 1 \\
    0 & 1 & 0 & 3 \\
    0 & 3 & 4 & 0 \\
    1 & 0 & 2 & 1
\end{pmatrix}.$$
Clearly, $A$ is grounded and $2$-increasing.
Since the mesh is equidistant and equal in both directions, the constant $\alpha_A$ is equal to $4$ times the maximal sum of a row or column of $S$, i.e., $\alpha_A=4 \cdot 8=32$.
An easy but somewhat lengthy calculation using \eqref{eq:underlineC} and \eqref{eq:overlineC} gives functions
$$\underline{C}_A=
\frac{1}{32}\begin{pmatrix}
 0 & 8 & 16 & 24 & 32 \\
 0 & 4 & 11 & 17 & 24 \\
 0 & 1 & 6 & 12 & 16 \\
 0 & 1 & 2 & 4 & 8 \\
 0 & 0 & 0 & 0 & 0 \\
\end{pmatrix}
\quad\text{and}\quad
\overline{C}_A=
\frac{1}{32}\begin{pmatrix}
 0 & 8 & 16 & 24 & 32 \\
 0 & 6 & 13 & 19 & 24 \\
 0 & 6 & 9 & 15 & 16 \\
 0 & 5 & 5 & 7 & 8 \\
 0 & 0 & 0 & 0 & 0 \\
\end{pmatrix},
$$
with distribution of mass
$$
\underline{S}=
\frac{1}{32}\begin{pmatrix}
 4 & 1 & 2 & 1 \\
 3 & 2 & 0 & 3 \\
 0 & 4 & 4 & 0 \\
 1 & 1 & 2 & 4 \\
\end{pmatrix}
\quad\text{and}\quad
\overline{S}=
\frac{1}{32}\begin{pmatrix}
 2 & 1 & 2 & 3 \\
 0 & 4 & 0 & 4 \\
 1 & 3 & 4 & 0 \\
 5 & 0 & 2 & 1 \\
\end{pmatrix}.
$$
It is now easy to verify that  $(\alpha_A \underline{S})_{ij} \ge S_{ij}$ and $(\alpha_A \overline{S})_{ij} \ge S_{ij}$ for every $i,j$, as claimed by Proposition~\ref{prop:2}.

\end{example}

Proposition~\ref{prop:2} shows that the mass of a nonzero grounded $2$-increasing function $A$ can be dominated by the mass of a multiple of a discrete copula, in particular, $\alpha_A\underline{C}_A$ and $\alpha_A \overline{C}_A$.
Next theorem shows that $\alpha_A$ is the optimal constant for such domination and establishes $\underline{C}_A$ and $ \overline{C}_A$ as the lower and upper bound for the corresponding discrete copulas.

\begin{theorem}\label{thm:2}
Let $A\colon \delta_1\times\delta_2\to\R$ be a nonzero grounded $2$-increasing function.
Suppose a real number $\alpha$  and a discrete copula $C\colon \delta_1\times\delta_2\to\I$ satisfy the condition $\alpha V_{C}(R_{ij})\ge V_A(R_{ij})$ for all 
$i \in \{1,2,\ldots,n\}$ and $j \in \{1,2,\ldots,m\}$. Then the following holds.
\begin{enumerate}[$(i)$]
\item $\alpha\ge \alpha_A $.
\item If $\alpha=\alpha_A $, then $\underline{C}_A \le C\le \overline{C}_A$.
\end{enumerate}
\end{theorem}

\begin{proof}
Clearly $\alpha V_C(R) \geq V_A(R)$ for all rectangles with vertices in $\delta_1 \times \delta_2$. Since $C$ is a discrete copula we have $V_C(\I\times[y_j,y_{j+1}])=(y_{j+1}-y_j).$ Therefore,
$$\alpha (y_{j+1}-y_j)= \alpha V_C(\I\times[y_j,y_{j+1}])\ge V_A(\I\times[y_j,y_{j+1}]). $$ Similarly also $\alpha (x_{i+1}-x_i)\ge V_A([x_i,x_{i+1}]\times \I)$. By the definition of $\alpha_A$ this implies
$\alpha\ge \alpha_A$, which proves the first claim.

Assume now that $\alpha=\alpha_A$. To prove the left inequality in item $(ii)$ we estimate
\begin{equation}\label{eq:lower_1}
\begin{aligned}
\alpha C(x_i,y_j)&=\alpha V_C([0,x_i] \times [0,y_j])\\
&\geq V_A([0,x_i] \times [0,y_j]) =\alpha_A \cdot\frac{V_A([0,x_i] \times [0,y_j])}{\alpha_A}
\end{aligned}
\end{equation}
by the fact that $\alpha V_C(R) \geq V_A(R)$ for all rectangles with vertices in $\delta_1\times\delta_2$.
Furthermore, using this fact again, we obtain
\begin{equation}\label{eq:lower_2}
\begin{aligned}
\alpha C(x_i,y_j) &=\alpha V_C([0,x_i] \times [0,y_j]) \\
&=\alpha \big(V_C(\I \times \I)-V_C(\I \times [y_j,1])-V_C([x_i,1] \times \I)+V_C([x_i,1] \times [y_j,1])\big)\\
&=\alpha\big(1-(1-y_j)-(1-x_i)\big)+\alpha V_C([x_i,1] \times [y_j,1])\\
&\geq \alpha(x_i+y_j-1)+V_A([x_i,1] \times [y_j,1])\\
&= \alpha_A(x_i+y_j-1)+V_A([x_i,1] \times [y_j,1])\\
&=\alpha_A \cdot\Bigl(x_i+y_j-1+\frac{V_A([x_i,1] \times [y_j,1])}{\alpha_A}\Bigr).
\end{aligned}
\end{equation}
Since $\alpha=\alpha_A$, then the inequalities~\eqref{eq:lower_1} and \eqref{eq:lower_2} imply
$$\begin{aligned}
\alpha_A C(x_i,y_j)  &\geq \max \bigg\{ \alpha_A \cdot\frac{V_A([0,x_i] \times [0,y_j])}{\alpha_A}, \alpha_A \cdot \Bigl(x_i+y_j-1+\frac{V_A([x_i,1] \times [y_j,1])}{\alpha_A}\Bigr) \bigg\}\\
&=\alpha_A \underline{C}_A(x_i,y_j).
\end{aligned}$$
Since this holds for all $i$ and $j$ we conclude that $C \ge \underline{C}_A$.

For the second inequality in item $(ii)$ we use similar arguments as above and estimate
\begin{equation}\label{eq:upper_1}
    \begin{aligned}
    \alpha C(x_i,y_j) &= \alpha V_C([0,x_i] \times [0,y_j])\\
    &=\alpha \big(V_C([0,x_i] \times \I)-V_C([0,x_i] \times [y_j,1])\big)\\
    &=\alpha x_i-\alpha V_C([0,x_i] \times [y_j,1]) \\
    &\le \alpha x_i-V_A([0,x_i] \times [y_j,1])\\
    &=\alpha_A x_i-V_A([0,x_i] \times [y_j,1])\\
    &=\alpha_A \cdot \Big(x_i-\frac{V_A([0,x_i] \times [y_j,1])}{\alpha_A}\Big)
    \end{aligned}
\end{equation}
and similarly
\begin{equation}\label{eq:upper_2}
    \alpha C(x_i,y_j) \le \alpha_A \cdot \Big(y_j-\frac{V_A([x_i,1] \times [0,y_j])}{\alpha_A}\Big).
\end{equation}
Since $\alpha=\alpha_A$, inequalities \eqref{eq:upper_1} and \eqref{eq:upper_2} imply $$\begin{aligned}
C(x_i,y_j) &\leq \min \bigg\{ x_i-\frac{V_A([0,x_i] \times [y_j,1])}{\alpha_A}, y_j-\frac{V_A([x_i,1] \times [0,y_j])}{\alpha_A}
\bigg\}=\overline{C}_A(x_i,y_j).
\end{aligned}$$
This shows that $\underline{C}_A \le C\le \overline{C}_A$, as claimed.
\end{proof}

A natural question arises whether the assumption $\alpha=\alpha_A$ in item $(ii)$ of Theorem~\ref{thm:2} is necessary. 
The next example shows that it is indeed essential. If a different, non-optimal constant is chosen, the corresponding discrete copulas are no longer bounded by $\underline{C}_A$ and $\overline{C}_A$.

\begin{example}
Let $\delta_1=\delta_2=\{0,\frac12,1\}$ and let $A \colon \delta_1 \times \delta_2 \to \R$ be a grounded function, whose mass distribution is given by the matrix
$$\begin{pmatrix}
    1 & 1 \\
    1 & 1
\end{pmatrix}.$$
Using equations \eqref{eq:alpha}, \eqref{eq:underlineC} and \eqref{eq:overlineC} we get
$\alpha_A=4$ and the distribution of mass of both $\underline{C}_A$ and $\overline{C}_A$ is given by the matrix
$$\begin{pmatrix}
    1/4 & 1/4 \\
    1/4 & 1/4
\end{pmatrix},$$
so that $\underline{C}_A=\overline{C}_A=\Pi$, the discrete product copula.
Let $C_1$ and $C_2$ be discrete copulas defined on $\delta_1 \times \delta_2$ with distribution of mass given respectively by matrices
$$\begin{pmatrix}
    2/6 & 1/6 \\
    1/6 & 2/6
\end{pmatrix}
\quad\text{and}\quad
\begin{pmatrix}
    1/6 & 2/6 \\
    2/6 & 1/6
\end{pmatrix}.
$$
Then $6V_{C_1}(R) \ge V_A(R)$ and $6V_{C_2}(R) \ge V_A(R)$ for all rectangles $R$ of the mesh, but $C_1<\Pi<C_2$.
\end{example}

Lemma~\ref{lem:1} expresses a discrete quasi-copula as an affine combination of two discrete copulas. This decomposition is not unique. For example, if $Q=\alpha_1 C_1+\alpha_2 C_2$ is a proper quasi-copula, then one of the coefficients must be negative, say $\alpha_2<0$, otherwise $Q$ would be a convex combination of copulas, which is again a copula.
In this case we can choose an arbitrary copula $C_3$ along with a real number $\alpha_3>0$ and express $Q$ also as
\begin{equation}\label{eq:alt}
    Q=(\alpha_1 C_1+\alpha_3 C_3)+(\alpha_2 C_2-\alpha_3 C_3)=\alpha_1' C_1'+\alpha_2' C_2',
\end{equation}
where $\alpha_1'=\alpha_1+\alpha_3>0$, $\alpha_2'=\alpha_2-\alpha_3<0$,
$$C_1'=\frac{\alpha_1}{\alpha_1+\alpha_3} C_1+\frac{\alpha_3}{\alpha_1+\alpha_3} C_3$$
and
$$C_2'=\frac{\alpha_2}{\alpha_2-\alpha_3} C_2+\frac{-\alpha_3}{\alpha_2-\alpha_3} C_3.$$
Since $C_1'$ and $C_2'$ are convex combinations of copulas, they are copulas themselves. So equality \eqref{eq:alt} gives an alternative decomposition of $Q$ as an affine combination of two copulas, one in which $|\alpha_1'|+|\alpha_2'|>|\alpha_1|+|\alpha_2|$.
Hence, it is reasonable to look for a decomposition of $Q$ in which the expression $|\alpha_1|+|\alpha_2|$ is as small as possible. In fact, the smallest value is the Minkowski norm of $Q$.
We note that while the minimal value of $|\alpha_1|+|\alpha_2|$ is unique, the minimizing values for $\alpha_1$ and $\alpha_2$ are unique only if $Q$ is a proper discrete quasi-copula.
If $Q$ is a copula, then $Q=tQ+(1-t)Q$ for any $t \in \R$, so the expression $|\alpha_1|+|\alpha_2|=|\alpha_1|+|1-\alpha_1|$ attains its minimal value $1$ for all $\alpha_1=t \in [0,1]$.
Nevertheless, the minimizing values for $\alpha_1$ and $\alpha_2$ are unique if we assume $\alpha_1 \ge 0$ and $\alpha_2 \le 0$.

By Lemma~\ref{lem:1} every discrete quasi-copula is contained in $\mathcal{S}$, the linear span of discrete copulas, so we may calculate its Minkowski norm.
With the above results we can now express the Minkowski norm of a discrete quasi-copula.
For a discrete quasi-copula $Q$ defined on $\delta_1 \times \delta_2$ let $Q_{\mathrm{pos}}$ and $Q_{\mathrm{neg}}$ be grounded $2$-increasing functions defined on $\delta_1 \times \delta_2$ such that
$$V_{Q_{\mathrm{pos}}}(R_{ij})=\max\{V_Q(R_{ij}),0\} \quad\text{and}\quad V_{Q_{\mathrm{neg}}}(R_{ij})=-\min\{V_Q(R_{ij}),0\}$$
for all rectangles $R_{ij}$ of the mesh $\delta_1 \times \delta_2$.
We remark that $Q_{\mathrm{pos}}$ and $Q_{\mathrm{neg}}$ are uniquely determined since they are grounded, and $Q=Q_{\mathrm{pos}} -Q_{\mathrm{neg}}$.

The proof of the next corollary relies on the fact, already mentioned in the Preliminaries, that the Minkowski norm~(see~\eqref{eq:Minkowski}) of a quasi-copula is always achieved (the proof for  general quasi-copulas can be found in \cite[Lemma~3.2]{DarOls95},  but  it works  also in the discrete case).

\begin{corollary}\label{cor:minkowski}
For a discrete quasi-copula $Q$ we have
$\alpha_{Q_{\mathrm{pos}}}-\alpha_{Q_{\mathrm{neg}}}=1$ and
$\|Q\|_M=2\alpha_{Q_{\mathrm{pos}}}-1$, where $\alpha_{Q_{\mathrm{pos}}}$ and $\alpha_{Q_{\mathrm{neg}}}$ are defined by \eqref{eq:alpha}.
\end{corollary}

\begin{proof}
If $Q$ is a discrete copula, then $Q_{\mathrm{neg}}$ is the zero function, $Q=Q_{\mathrm{pos}}$, $\alpha_{Q_{\mathrm{neg}}}=0$, $\alpha_{Q_{\mathrm{pos}}}=\alpha_Q=1$, and
$\|Q\|_M=1=2\alpha_{Q_{\mathrm{pos}}}-1$. In this case the claim holds. We assume in the rest of the proof that $Q$ is a proper discrete copula, so that both $Q_{\mathrm{pos}}$ and $Q_{\mathrm{neg}}$ are nonzero functions.

The equality $\alpha_{Q_{\mathrm{pos}}}-\alpha_{Q_{\mathrm{neg}}}=1$ follows directly from the definition given in \eqref{eq:alpha} because
$$V_{Q_{\mathrm{pos}}}([x_i,x_{i+1}] \times \I)-V_{Q_{\mathrm{neg}}}([x_i,x_{i+1}] \times \I)=V_Q([x_i,x_{i+1}] \times \I)=x_{i+1}-x_i$$
for all $i$, and a similar equality holds for the set $\I \times [y_j,y_{j+1}]$.

To prove the second claim, let $\|Q\|_M=s+t$ where $Q=sA-tB$ for some discrete copulas $A$ and $B$ (cf. \cite[Lemma~3.2]{DarOls95}, the proof in the discrete case is the same).
Then $V_{Q_{\mathrm{pos}}}(R_{ij})-V_{Q_{\mathrm{neg}}}(R_{ij})=sV_A(R_{ij})-tV_B(R_{ij})$.
For each $i$ and $j$ only one term on the left-hand side is nonzero, hence,
$V_{Q_{\mathrm{pos}}}(R_{ij}) \le sV_A(R_{ij})$ and $V_{Q_{\mathrm{neg}}}(R_{ij}) \le tV_B(R_{ij})$ for all $i$ and $j$.
By Theorem~\ref{thm:2} we have $s \ge \alpha_{Q_{\mathrm{pos}}}$ and $t \ge \alpha_{Q_{\mathrm{neg}}}=\alpha_{Q_{\mathrm{pos}}}-1$, so that $\|Q\|_M=s+t \ge 2\alpha_{Q_{\mathrm{pos}}}-1$. It remains to prove the opposite inequality.

Since $Q$ is a proper discrete quasi-copula, it follows that $\alpha_{Q_{\mathrm{pos}}}-1=\alpha_{Q_{\mathrm{neg}}}>0$.
Note that
$Q=\alpha_{Q_{\mathrm{pos}}}\overline{C}_{Q_{\mathrm{pos}}}-(\alpha_{Q_{\mathrm{pos}}}-1)B$, where
$$B=\frac{1}{\alpha_{Q_{\mathrm{pos}}}-1} (\alpha_{Q_{\mathrm{pos}}}\overline{C}_{Q_{\mathrm{pos}}}-Q).$$
Proposition~\ref{prop:2} implies that $\overline{C}_{Q_{\mathrm{pos}}}$ is a copula, and $\alpha_{Q_{\mathrm{pos}}}V_{\overline{C}_{Q_{\mathrm{pos}}}}(R_{ij}) \ge V_{Q_{\mathrm{pos}}}(R_{ij}) \ge V_Q(R_{ij})$, so that $B$ is a copula as well (it is clearly grounded and has uniform marginals).
This implies
$$\|Q\|_M \le \alpha_{Q_{\mathrm{pos}}}+(\alpha_{Q_{\mathrm{pos}}}-1)=2\alpha_{Q_{\mathrm{pos}}}-1,$$
which finishes the proof.
\end{proof}

While every discrete quasi-copula is a linear combination of discrete copulas, it was noted in \cite{FerQueUbe21} that there are quasi-copulas defined on $\I^2$ that cannot be written as linear combinations of copulas. This is because a quasi-copula that can be written as a linear combination of copulas, can be written as a linear combination of two copulas, and as such, it induces a signed measure on the Borell $\sigma$-algebra in $\I^2$. However, there exist quasi-copulas that do not induce a signed measure on this $\sigma$-algebra, see \cite{FerRodUbe11}.
We investigate linear combinations of copulas defined on $\I^2$ in our next section, see Theorem~\ref{thm:absolute-convergence} and Example~\ref{exa:final}.

\section{Quasi-copulas as converging sums}\label{sec:general}

In this section we investigate decompositions of copulas
as finite and infinite linear combinations of copulas. By an infinite linear combination we mean an infinite sum of multiples of copulas that converges in the supremum norm.

Our first theorem of this section shows that every quasi-copula is an infinite linear combination of copulas. In fact, there are many such decompositions. We will only be interested in their existence and will leave aside the question of their optimality. Note that the converse is false because even a linear combination of two copulas is not necessarily  a quasi-copula, nor a multiple of a quasi-copula. We rely heavily on the fact that every discrete quasi-copula can be extend to a quasi-copula defined on the entire unit square (see \cite[Proposition~4]{OmlSto20} and also~\cite{DurFerTru16}).
\begin{theorem}\label{thm:series}
For every quasi-copula $Q$ there exist copulas $\{C_j\}_{j=1}^\infty$ and real numbers $\{\gamma_j\}_{j=1}^\infty$ with $\sum_{j=1}^\infty \gamma_j =1$ such that
$$Q(x,y)=\sum_{j=1}^\infty \gamma_j C_j(x,y) \qquad
\text{for all $x,y \in \I$,}$$
the series converges uniformly, and all of its partial sums are positive multiples of quasi-copulas.
\end{theorem}
\begin{proof}
Let $n$ be an integer. Choose a finite mesh $\delta_1\times\delta_2\subseteq\I^2$ with $\delta_1=\{x_1,x_2,\ldots,x_{k+1}\}$ and $\delta_2=\{y_1,y_2,\ldots,y_{m+1}\}$ such that
$$|x_{i+1}-x_i|<\frac{1}{n} \quad\text{and}\quad |y_{j+1}-y_j|<\frac{1}{n}$$
for all $i \in \{1,2,\ldots,k\}$ and $j \in \{1,2,\ldots,m\}$.
Let $Q_n$ be the discrete quasi-copula defined  as the restriction of $Q$ to the mesh $\delta_1\times\delta_2$.  By Lemma~\ref{lem:1} there exist discrete copulas $A_n$ and $B_n$ and real numbers $\alpha_n \geq 1$ and $\beta_n \leq 0$ with $\alpha_n+\beta_n=1$ such that $$Q_n =\alpha_n A_n+\beta_n B_n.$$
We extend discrete copulas $A_n$ and $B_n$ piecewise bilinearly to obtain copulas $\widehat{A}_n$ and $\widehat{B}_n$ defined on $\I^2$. Then, $\widehat{Q}_n:=\alpha_n\widehat{A}_n+\beta_n \widehat{B}_n$ coincides with $Q$ on  the mesh $\delta_1\times\delta_2$ and is a piecewise bilinear extension of $Q_n$ to $\I^2$.  As such, $\widehat{Q}_n$ is a quasi-copula by \cite[Proposition~4]{OmlSto20}.
Take any $i\in \{1,2,\ldots,k\}$, $j \in \{1,2,\ldots,m\}$, and $(x,y)\in R_{ij}$. Since both $Q$ and $\widehat{Q}_n$ are $1$-Lipschitz, we can estimate
\begin{align*}
 |Q(x,y)-\widehat{Q}_n(x,y)|&\le |Q(x,y)-Q(x_i,y_j)| + |Q(x_i,y_j)-\widehat{Q}_n(x_i,y_j)|\\
 &\hspace{1cm}\mbox{}+|\widehat{Q}_n(x_i,y_j)-\widehat{Q}_n(x,y)| \\
 &=|Q(x,y)-Q(x_i,y_j)| + |\widehat{Q}_n(x_i,y_j)-\widehat{Q}_n(x,y)| \\
 &\le (|x-x_i|+|y-y_j|)+(|x_i-x|+|y_j-y|)\le \frac{4}{n}.
\end{align*}
This shows that the sequence of functions $\widehat{Q}_n$
converges uniformly to the quasi-copula $Q$.
Therefore, the series
\begin{equation*}\label{eq:Q}
\widehat{Q}_1 + \sum_{n=1}^\infty(\widehat{Q}_{n+1}-\widehat{Q}_n)
\end{equation*}
converges uniformly to $Q$.
Observe that $\widehat{Q}_1=\alpha_{1} \widehat{A}_{1}+\beta_{1} \widehat{B}_{1}$ and
\begin{equation}\label{eq:Qn+1-Qn}
\begin{aligned}
    \widehat{Q}_{n+1}-\widehat{Q}_n &=
    \bigl(\alpha_{n+1} \widehat{A}_{n+1}+\beta_{n+1} \widehat{B}_{n+1}\bigr)- \bigl(\alpha_{n} \widehat{A}_{n}+\beta_{n} \widehat{B}_{n}\bigr)\\
    &=\bigl(\alpha_{n+1} \widehat{A}_{n+1}-\beta_{n} \widehat{B}_{n}\bigr) + \bigl(\beta_{n+1} \widehat{B}_{n+1}-\alpha_{n} \widehat{A}_{n}\bigr)\\
    &=\zeta_n D_n+\xi_n E_n,
\end{aligned}
\end{equation}
where $\zeta_n:=\alpha_{n+1}-\beta_n\ge  1$,  $\xi_n:=\beta_{n+1}-\alpha_n\le -1$, 
$$D_n:=\frac{\alpha_{n+1}}{\alpha_{n+1}-\beta_n}\widehat{A}_{n+1}+\frac{(-\beta_n)}{\alpha_{n+1}-\beta_n}\widehat{B}_{n}$$ and 
$$E_n:=\frac{\beta_{n+1}}{\beta_{n+1}-\alpha_n}\widehat{B}_{n+1}+\frac{(-\alpha_n)}{\beta_{n+1}-\alpha_n}\widehat{A}_{n}.$$
Since $D_n$ and $E_n$ are convex combinations of copulas they are copulas themselves.
By the above results, we have that
\begin{equation}\label{eq:QQ}
 Q(x,y)=\alpha_{1} \widehat{A}_{1}(x,y)+\beta_{1} \widehat{B}_{1}(x,y)+\sum_{n=1}^\infty \big(\zeta_n D_n(x,y)+\xi_n E_n(x,y)\big) 
\end{equation}
for all $(x,y) \in \I^2$ and the sum converges uniformly. 

However, observe that if the parenthesis in the series \eqref{eq:QQ} are omitted, the resulting series  $$\zeta_1 D_1(x,y)+\xi_1 E_1(x,y)+\zeta_2 D_2(x,y)+\xi_2 E_2(x,y)+\zeta_3D_3(x,y)+\dots ,$$
 does not converge, e.g., at $(x,y)=(1,1)$ we obtain a  series
$\zeta_1+\xi_1+\zeta_2+\xi_2+\zeta_3+\cdots$ which is divergent 
because $\zeta_n\ge 1$. 
To finish the proof
of  Theorem~\ref{thm:series} we will modify the series \eqref{eq:QQ}  in such a way that it will converge even without parenthesis.

For every $n$ we fix an  integer $K_n$, such that $K_n >|\xi_n|$. Note that, by \eqref{eq:QQ},
\begin{equation}\label{eq:modified}
\begin{aligned}
Q(x,y)&= \alpha_{1} \widehat{A}_{1}(x,y)+\beta_{1} \widehat{B}_{1}(x,y)+\sum_{n=1}^
\infty\sum_{i=1}^{nK_n} \frac{1}{n K_n}\bigl(\zeta_n  D_n(x,y) +\xi_n  E_n(x,y)\bigr) 
\end{aligned}
\end{equation}
Let us prove that the expanded series
\begin{equation}\label{eq:final}
\begin{aligned}
&\alpha_{1} \widehat{A}_{1}(x,y)+\beta_{1} \widehat{B}_{1}(x,y)+\\
&\mbox{}+
\underbrace{\frac{\zeta_1}{ K_1} D_1(x,y)+\frac{\xi_1}{ K_1} E_1(x,y) +\dots+\frac{\zeta_1}{ K_1} D_1(x,y)+\frac{\xi_1}{ K_1} E_1(x,y)}_{2K_1\text{ terms}} +\\
&\mbox{}+
\underbrace{\frac{\zeta_2}{ 2K_2} D_2(x,y)+\frac{\xi_2}{2 K_2} E_2(x,y) +\dots+\frac{\zeta_2}{2 K_2} D_2(x,y)+\frac{\xi_2}{2 K_2} E_2(x,y)}_{4K_2\text{ terms}} +\\
&\mbox{}+\underbrace{\frac{\zeta_3}{ 3K_3} D_3(x,y)+\frac{\xi_3}{3 K_3} E_3(x,y) +\dots+\frac{\zeta_3}{3 K_3} D_3(x,y)+\frac{\xi_3}{3 K_3} E_3(x,y)}_{6K_3\text{ terms}} +\\
&\mbox{}+\frac{\zeta_4}{ 4K_4} D_4(x,y)+\dots
\end{aligned}
\end{equation}
converges uniformly to $Q(x,y)$. Choose $\varepsilon>0$. Note that the series \eqref{eq:modified}, which  equals the series \eqref{eq:QQ} when its inner finite sums  are calculated,  does converge uniformly to $Q(x,y)$. So if $p$ is sufficiently  large, then its remainder
\begin{equation}\label{eq:first-esst}
    \bigg|\sum_{n=p+1}^
\infty\sum_{i=1}^{nK_n} \frac{1}{n K_n}\bigl(\zeta_n  D_n(x,y) +\xi_n  E_n(x,y)\bigr)\bigg|<\varepsilon \quad\text{for all $x,y \in \I$}.
\end{equation}
 In addition, being a term of a convergent series \eqref{eq:QQ} we also have 
 $$|\zeta_p  D_p(x,y) +\xi_p  E_p(x,y)|<\varepsilon \quad \text{for all $x,y \in \I$}.$$
Hence, for all $k \le pK_p$ and all $x,y \in \I$,
\begin{equation}\label{eq:divided}\bigg|\sum_{i=1}^{k} \frac{1}{p K_p}\bigl(\zeta_p  D_p(x,y) +\xi_p  E_p(x,y)\bigr)\bigg|=\frac{k}{pK_p}|\zeta_p  D_p(x,y) +\xi_p  E_p(x,y)|<\varepsilon.
\end{equation}
Note that each partial sum of the series \eqref{eq:final} is either of the form
\begin{equation}\label{eq:psum1}
\begin{aligned}
\alpha_{1} \widehat{A}_{1}(x,y)+\beta_{1} \widehat{B}_{1}(x,y)&+\sum_{n=1}^
{p-1}\sum_{i=1}^{nK_n} \frac{1}{n K_n}\bigl(\zeta_n  D_n(x,y) +\xi_n  E_n(x,y)\bigr)\\
&+\sum_{i=1}^{k} \frac{1}{p K_p}\bigl(\zeta_p  D_p(x,y) +\xi_p  E_p(x,y)\bigr),
\end{aligned}\end{equation}
for some $p\ge 1$ and $k \in \{1,2,\ldots,pK_p\}$ or of the form
\begin{equation}\label{eq:psum2}\begin{aligned}
\alpha_{1} \widehat{A}_{1}(x,y)&+\beta_{1} \widehat{B}_{1}(x,y)
+\sum_{n=1}^
{p-1}\sum_{i=1}^{nK_n} \frac{1}{n K_n}\bigl(\zeta_n  D_n(x,y) +\xi_n  E_n(x,y)\bigr)\\
&+\sum_{i=1}^{k-1} \frac{1}{p K_p}\bigl(\zeta_p  D_p(x,y) +\xi_p  E_p(x,y)\bigr)+\frac{1}{p K_p}\zeta_p  D_p(x,y),
\end{aligned}\end{equation}
for some $p\ge 1$ and $k \in \{1,2,\ldots,pK_p\}$.
For $p$ sufficiently large, the distance between $Q$ and partial sum \eqref{eq:psum1} can be bounded from above, using the triangular inequality, \eqref{eq:divided}, and \eqref{eq:first-esst}, by the expression
\begin{align*}
&\bigg|\sum_{i=k+1}^{pK_p}\frac{1}{p K_p}\bigl(\zeta_p  D_p(x,y) +\xi_p  E_p(x,y)\bigr)\bigg|\\
&+\bigg|\sum_{n=p+1}^
\infty\sum_{i=1}^{nK_n} \frac{1}{n K_n}\bigl(\zeta_n  D_n(x,y) +\xi_n  E_n(x,y)\bigr) \bigg|<\varepsilon+\varepsilon.
\end{align*}
Similarly, using also $K_p \ge |\xi_p|$, the distance between $Q$ and partial sum \eqref{eq:psum2} can be bounded from above by 
\begin{align*}
&\bigg|\frac{1}{pK_p}\xi_pE_p(x,y)\bigg|+\bigg|\sum_{i=k+1}^{pK_p}\frac{1}{p K_p}\bigl(\zeta_p  D_p(x,y) +\xi_p  E_p(x,y)\bigr)\bigg|\\
&+\bigg|\sum_{n=p+1}^
\infty\sum_{i=1}^{nK_n} \frac{1}{n K_n}\bigl(\zeta_n  D_n(x,y) +\xi_n  E_n(x,y)\bigr) \bigg|<\frac{1}{p}+\varepsilon+\varepsilon.
\end{align*}
This shows that the partial sums of \eqref{eq:final} uniformly converge to $Q(x,y)$.

It only remains to prove that the partial sums of the series \eqref{eq:final} are positive multiples of quasi-copulas.
To do that one observes, by \eqref{eq:Qn+1-Qn} and \eqref{eq:QQ}, that the partial sums either take the form
\begin{equation}\label{eq:first-case}
\begin{aligned}
\widehat{Q}_p(x,y)&+\sum_{i=1}^{k} \frac{1}{p K_p}\bigl(\zeta_p  D_p(x,y) +\xi_p  E_p(x,y)\bigr)\\
&=\widehat{Q}_p(x,y)+\frac{k}{pK_p} (\widehat{Q}_{p+1}(x,y)-\widehat{Q}_{p}(x,y))\\
&=(1-\tfrac{k}{p K_p})\widehat{Q}_p(x,y)+\tfrac{k}{pK_p}\widehat{Q}_{p+1}(x,y)  
\end{aligned}
\end{equation}
or else they take the form
\begin{equation}\label{eq:second-case}
\begin{aligned}
\widehat{Q}_p(&x,y)+\sum_{i=1}^{k-1} \frac{1}{p K_p}\bigl(\zeta_p  D_p(x,y) +\xi_p  E_p(x,y)\bigr)+\frac{\zeta_p}{p K_p}  D_p(x,y)\\
&=(1-\tfrac{k-1}{p K_p})\widehat{Q}_p(x,y)+\tfrac{k-1}{pK_p}\widehat{Q}_{p+1}(x,y)+\tfrac{\zeta_p}{p K_p}  D_p(x,y)\\
&=(1+\tfrac{\zeta_p}{pK_p})\Bigl(\tfrac{ pK_p-k+1}{pK_p+\zeta_p}\widehat{Q}_p(x,y)+\tfrac{k-1}{pK_p+\zeta_p}\widehat{Q}_{p+1}(x,y)+\tfrac{\zeta_p}{p K_p+\zeta_p}  D_p(x,y)\Bigr)
\end{aligned}
\end{equation}
for some $p\ge 1$ and some $k\in\{1,\dots,p K_p\}$. Recall from the start of the proof that $\widehat{Q}_p$ is a quasi-copula, and so are $D_p$ and $E_p$. Hence,  the partial sum \eqref{eq:first-case}, being  a convex combination of quasi-copulas, is a quasi-copula itself. Similarly, the partial sum~\eqref{eq:second-case} is a $(1+\tfrac{\zeta_p}{pK_p})$ multiple of a  quasi-copulas. The result follows since, as shown at the beginning of the proof,  $\zeta_p\ge 1$.
\end{proof}

Note that every partial sum of the series in Theorem~\ref{thm:series} belongs to $\mathcal{S}$, the linear span of copulas. Since the sum converges uniformly, we immediately obtain the following corollary, which extends \cite[Corollary~11]{FerQueUbe21} to the non-discrete setting.

\begin{corollary}
The closure of $\mathcal{S}$ in the supremum norm contains all quasi-copulas.
\end{corollary}

We now generalize definition \eqref{eq:alpha} from discrete quasi-copulas to general quasi-copulas~$Q$ as follows:
\begin{equation}\label{eq:alpha_general}
    \alpha_Q=\displaystyle\sup_{n \geq 1} \left\{ \max_{i=1,\ldots,2^n}2^n \sum_{j=1}^{2^n} V_Q(R_{ij}^{n})^+ \,,\, \max_{j=1,\ldots,2^n}2^n \sum_{i=1}^{2^n} V_Q(R_{ij}^{n})^+ \right\},
\end{equation}
where $R_{ij}^{n}=[\frac{i-1}{2^n},\frac{i}{2^n}] \times [\frac{j-1}{2^n},\frac{j}{2^n}]$, $1\le i,j\le 2^n$ and where we denote
$x^+=\max\{x,0\}$ and $x^-=-\min\{x,0\}$  for any real number $x$.
In our next theorem we characterize quasi-copulas that lie in $\mathcal{S}$.

\begin{theorem}\label{thm:absolute-convergence}
Let $Q$ be a quasi-copula. The following conditions are equivalent:
\begin{enumerate}[$(i)$]
\item There exist copulas $\{C_j\}_{j=1}^\infty$ and real numbers $\{\gamma_j\}_{j=1}^\infty$ such that
$$Q(x,y)=\sum_{j=1}^\infty \gamma_j C_j(x,y),$$
where the series converges absolutely for all $x,y \in \I$.
\item $Q \in \mathcal{S}$, i.e., there exist copulas $A$ and $B$ and real numbers $\alpha$ and $\beta$ such that
$Q(x,y)=\alpha A(x,y) + \beta B(x,y)$ for all $x,y \in \I$.
\item We have $\alpha_Q < \infty$.
\end{enumerate}
\end{theorem}

\begin{proof}
$(i) \Longrightarrow (ii)$: Since the series converges absolutely we may  collect together the terms with $\gamma_j>0$ and the terms with $\gamma_j<0$. Note that, due to $Q(1,1)=1$, at least one of $\gamma_j$ must be  positive.
Therefore, $A(x,y):=\frac{1}{\sum_{j=1}^\infty \gamma_j^+}\sum_{j=1}^\infty \gamma_j^+ C_j(x,y)$ is a well-defined (possibly infinite) convex combination of copulas $C_j(x,y)$ so is itself a copula.
If each $\gamma_j\ge0$, then $Q(x,y)=\alpha A(x,y)$ for $\alpha:=\sum_{j=1}^\infty \gamma_j^+$ and $(ii)$ holds by taking $\beta=0$ and taking any copula for $B$.
Otherwise, $(ii)$ holds if we define a copula $B(x,y):=\frac{1}{\sum_{j=1}^\infty \gamma_j^-}\sum_{j=1}^\infty\gamma_j^- C_j(x,y)$ and $\beta=-\sum_{j=1}^\infty \gamma_j^-$.

$(ii) \Longrightarrow (i)$: Trivial.

$(ii) \Longrightarrow (iii)$: Clearly  $\alpha+\beta=1$ since $Q(1,1) = A(1,1) = B(1,1) = 1$.  We can  assume that  $\alpha\ge 1$ and $\beta\le 0$. Namely, if both $\alpha,\beta\ge0$, then 
$Q$ is a convex combination of two copulas, so already a  copula, hence we can redefine $A=Q$ and take $\alpha=1$ and  $\beta=0$. 

Recall that copulas $A$ and $B$ induce positive measures $\mu_A$ and $\mu_B$ on Borel subsets of $\I^2$, so $Q$ induces a signed measure $\mu_Q=\alpha \mu_A+\beta \mu_B$.
By the Jordan decomposition we can express it as a difference $\mu_Q=\mu^+-\mu^-$ of two positive measures $\mu^+$ and $\mu^-$ with mass concentrated on disjoint  sets $P^+,P^-\subseteq\I^2$. Choose any rectangle $S_i=[\frac{i-1}{2^n},\frac{i}{2^n}]\times\I$. Then
\begin{align*}
\mu^+(S_i) &=
\mu_Q(P^+\cap S_i)\le \alpha \mu_A(P^+\cap S_i)\le \alpha \mu_A(S_i)=\alpha \cdot\tfrac{1}{2^n}.
\end{align*}
It follows that  $\alpha\ge 2^n \mu^+(S_i)$ for every positive integer $n$ and every $i \in \{1,\ldots,2^n\}$. 
Clearly, $$V_Q(R_{ij}^n)^+=\big(\mu^+(R_{ij}^n)-\mu^-(R_{ij}^n)\big)^+\le \mu^+(R_{ij}^n)$$
so $\sum_{j=1}^{2^n} V_Q(R_{ij}^n)^+\le \sum_{j=1}^{2^n} \mu^+(R_{ij}^{n})=\mu^+(S_i) $.  This implies that
$$\alpha\ge \max_{i=1,\ldots,2^n} 2^n\mu^+(S_i)\ge \max_{i=1,\ldots,2^n} 2^n \sum_{j=1}^{2^n} V_Q(R_{ij}^n)^+$$
for each $n$. Similarly, by considering rectangles $T_j=\I\times [\frac{j-1}{2^n},\frac{j}{2^n}]$ we see that 
$$\alpha\ge \max_{j=1,\ldots,2^n} 2^n\mu^+(T_j)\ge \max_{j=1,\ldots,2^n} 2^n \sum_{i=1}^{2^n} V_Q(R_{ij}^n)^+$$
for every $n$, and the result follows.

$(iii) \Longrightarrow (ii)$:
For a positive integer $n$ let $\delta_n=\{0,\frac{1}{2^n},\frac{2}{2^n},\ldots,1\}$. Denote
\begin{equation}\label{eq:alpha_n}
    \alpha_{n}=\max_{\substack{i=1,\ldots,2^n\\ j=1,\ldots,2^n}} \left\{ 2^n \sum_{j=1}^{2^n} V_Q(R_{ij}^{n})^+ \,,\, 2^n \sum_{i=1}^{2^n}V_Q(R_{ij}^{n})^+ \right\}
\end{equation}
and let $A_{n}$ be a grounded discrete function defined on $\delta_n \times \delta_n$, such that $V_{A_{n}}(R_{ij}^{n})=V_Q(R_{ij}^{n})^+$, i.e.,
$$A_{n}(\tfrac{i}{2^n},\tfrac{j}{2^n})=\sum_{k=1}^i \sum_{l=1}^j V_Q(R_{kl}^{n})^+.$$
Clearly, $A_{n}$ is $2$-increasing and nonzero.
Note that
$$\begin{aligned}
\alpha_{n}&=\max_{\substack{i=1,\ldots,2^n\\ j=1,\ldots,2^n}} \left\{ 2^n\sum_{j=1}^{2^n}  V_{A_{n}}(R_{ij}^{n}) \,,\, 2^n\sum_{i=1}^{2^n} V_{A_{n}}(R_{ij}^{n}) \right\}\\
&=\max_{\substack{i=1,\ldots,2^n\\ j=1,\ldots,2^n}} \Bigl\{ 2^n V_{A_{n}}([\tfrac{i-1}{2^n},\tfrac{i}{2^n}] \times \I) \,,\, 2^nV_{A_{n}}(\I \times [\tfrac{j-1}{2^n},\tfrac{j}{2^n}]) \Bigr\}=\alpha_{A_{n}}
\end{aligned}$$
where $\alpha_{A_{n}}$ is defined as in \eqref{eq:alpha} with respect to $A=A_{n}$.
Observe that $V_{A_{n}}([\tfrac{i-1}{2^n},\tfrac{i}{2^n}]\times \I)\ge V_{Q}([\tfrac{i-1}{2^n},\tfrac{i}{2^n}]\times\I)=\frac{1}{2^n}$ where the last equality follows because $Q$ is a quasi-copula. Similarly, $V_{A_{n}}(\I\times [\tfrac{j-1}{2^n},\tfrac{j}{2^n}])\ge V_{Q}(\I\times [\tfrac{j-1}{2^n},\tfrac{j}{2^n}])=\frac{1}{2^n}$. Therefore,
\begin{equation}\label{eq:alpha(mn}
\alpha_{n}\ge 1.
\end{equation}

Define, as in \eqref{eq:overlineC},
$$C_{n}=\overline{C}_{A_{n}}.$$  By Proposition~\ref{prop:2} the function $C_{n}$ is a discrete copula and we have
\begin{equation}\label{eq:selena}
\alpha_{n} V_{C_{n}}(R_{ij}^{n}) \geq V_{A_{n}}(R_{ij}^{n})
\end{equation} for all $i$ and $j$.

We can extend the discrete copula $C_{n}$ to a piecewise bilinear copula $\widehat{C}_{n}\colon \I^2 \to \I$. By the hypothesis in item $(iii)$ the sequence of numbers $(\alpha_{n})_{n=1}^\infty$ is bounded, hence, by passing to a subsequence, we may assume it is convergent with limit $\alpha$.
Furthermore, since the set of all copulas is compact in the supremum norm, we may likewise assume that the sequence of copulas $(\widehat{C}_{n})_{n=1}^\infty$ converges to a copula $C \colon \I^2 \to \I$.
Note that, by \eqref{eq:alpha(mn}, $\alpha=\lim_{n \to \infty}\alpha_{n}\ge1$. We consider two cases

\medskip\noindent {\bf Case $\alpha=1$.} We claim that, in this case, $Q$ is a copula. To see this,  choose an arbitrary rectangle $R\subseteq\I^2$ and let $\varepsilon>0$. Given an integer $n$, let
\begin{equation}\label{eq:R_n'}
R'_n=[\tfrac{i_0}{2^n},\tfrac{i_1}{2^n}] \times [\tfrac{j_0}{2^n},\tfrac{j_1}{2^n}]
\end{equation}
be a maximal rectangle, with vertices on the mesh $\delta_n\times\delta_n$, such that $R'_n \subseteq R$.
Since quasi-copula $Q$ is continuous and $V_Q(R)$ is a sum and difference of the values of $Q$ at the four vertices of $R$, we see that if $n$ is sufficiently large (so that the mesh is sufficiently fine), then  
\begin{equation}\label{eq:abs}
|V_Q(R_n')-V_Q(R)|<\varepsilon.
\end{equation}
Due to $\alpha=\lim_{n \to \infty}\alpha_{n}=1$ we may, if needed, increase $n$ to achieve that $\alpha_{n}<1+\varepsilon$.
By definition of $\alpha_{n}$ we then have
\begin{equation}\label{eq:1eps}
2^n V_{A_{n}}([\tfrac{i-1}{2^n},\tfrac{i}{2^n}]\times\I)<1+\varepsilon.
\end{equation}
Note that we have
\begin{align*}
1=2^n V_{Q}([\tfrac{i-1}{2^n},\tfrac{i}{2^n}]\times\I)&=2^n\sum_{j=1}^{2^n} \left(V_Q(R_{ij}^{n})^+ -V_Q(R_{ij}^{n})^-\right)\\
&=2^n V_{A_{n}}([\tfrac{i-1}{2^n},\tfrac{i}{2^n}]\times\I)- 2^n\sum_{j=1}^{2^n} V_Q(R_{ij}^{n})^-, 
\end{align*}
 hence, by inequality \eqref{eq:1eps}, we get
$$\sum_{j=1}^{2^n} V_Q(R_{ij}^{n})^- =V_{A_{n}}([\tfrac{i-1}{2^n},\tfrac{i}{2^n}]\times\I)-\frac{1}{2^n}<\frac{\varepsilon}{2^n}$$
for every $i$. Summing over all $i$ we obtain
\begin{equation}\label{eq:estimate1}
\sum_{i=1}^{2^n}\sum_{j=1}^{2^n} V_Q(R_{ij}^{n})^- <\varepsilon.
\end{equation}
By \eqref{eq:abs} we have
\begin{align*}
V_Q(R) &>V_Q(R'_n)-\varepsilon=\sum_{i=i_0+1}^{i_1}\sum_{j=j_0+1}^{j_1} (V_Q(R_{ij}^{n})^+ -V_Q(R_{ij}^{n})^-) -\varepsilon\\
&\ge \sum_{i=i_0+1}^{i_1}\sum_{j=j_0+1}^{j_1} (0-V_Q(R_{ij}^{n})^-) -\varepsilon\ge-2\varepsilon,
\end{align*}
where the last inequality follows from \eqref{eq:estimate1}. Since $\varepsilon>0$ was arbitrary, we conclude that $V_Q(R)\ge 0$ and therefore $Q$ is a copula.

\medskip\noindent  {\bf Case $\alpha>1$.}
Define $D=\frac{1}{1-\alpha}(Q-\alpha C)$. 
We claim that  $D$ is a copula. Clearly, it is grounded and has uniform marginals. To see that it is also $2$-increasing we adjust the arguments from the previous case as follows.

Choose an arbitrary rectangle $R\subseteq\I^2$, let $\varepsilon>0$, and let $R'_n \subseteq R$ be as in \eqref{eq:R_n'}. Define $D_n=\frac{1}{1-\alpha_{n}}(Q-\alpha_{n}\widehat{C}_{n})$. Since  $D_n$ uniformly converges to $D$ and $D$ is  continuous we see that 
\begin{equation}\label{eq:boj}
    |V_{D_n}(R_n')-V_D(R)|\le |V_{D_n}(R_n')-V_D(R_n')|+|V_{D}(R_n')-V_D(R)|<\varepsilon
\end{equation}
for every $n$ sufficiently large.

By  inequality \eqref{eq:selena}  and the definition of $A_{n}$  we have
$$\alpha_{n} V_{\widehat{C}_{n}}(R_{n}')=\alpha_{n} V_{C_{n}}(R_{n}') \ge V_{A_{n}}(R_{n}')\ge V_{Q}(R_{n}')$$
for every positive integer $n$.
It follows that
$$(1-\alpha_{n})V_{D_n}(R_n')=V_Q(R_n')-\alpha_{n} V_{\widehat{C}_{n}}(R_{n}')\le 0$$
for every positive integer $n$.
Due to $\lim_{n \to \infty}(1-\alpha_{n})=1-\alpha<0$ it follows that, for each $n$ sufficiently large,
$V_{D_n}(R_n')\ge0$. Combined with \eqref{eq:boj} we see that 
$V_D(R)> -\varepsilon$. Sending $\varepsilon$ to $0$ we obtain $V_D(R)\ge 0$ and therefore $D$ is $2$-increasing, as claimed.
 
We conclude that $Q=\alpha C+(1-\alpha) D$ is a linear combination of copulas $C$ and $D$.
\end{proof}

\begin{remark}\label{rem:main}
By modifying the above arguments one can show that, in Theorem~\ref{thm:absolute-convergence}, items $(i)$ and $(ii)$ are equivalent to 
\begin{itemize}
\item[$(iii')$] If $\bigl(\delta_{n_k}\times\delta_{m_k}\bigr)_{k}=\bigl(\{x_1^{(k)},\dots, x_{n_k+1}^{(k)}\}\times\{y_1^{(k)},\dots,y_{m_k+1}^{(k)}\}\bigr)_{k}$ is any sequence of meshes in $\I^2$ satisfying
$$0=x_1^{(k)}<x_2^{(k)}<\ldots<x_{n_k+1}^{(k)}=1,$$
$$0=y_1^{(k)}<y_2^{(k)}<\ldots<y_{m_k+1}^{(k)}=1,$$
such that
$$\max_{i=1,\ldots,n_k}( x_{i+1}^{(k)}-x_{i}^{(k)})\xrightarrow{k\to\infty}{0}, \quad
\max_{i=1,\ldots,m_k} (y_{j+1}^{(k)}-y_{j}^{(k)})\xrightarrow{k\to\infty}{0},$$ then
$$\sup_{m_k,n_k \geq 1} \left\{ \max_{i=1,\ldots,n_k} \sum_{j=1}^{m_k} \frac{V_Q(R_{ij}^{(k)})^+}{ x_{i+1}^{(k)}-x_{i}^{(k)}} \,,\, \max_{j=1,\ldots,m_k} \sum_{i=1}^{n_k} \frac{V_Q(R_{ij}^{(k)})^+}{ y_{j+1}^{(k)}-y_{j}^{(k)}} \right\}<\infty,$$
where $R_{ij}^{(k)}=[x_{i}^{(k)},x_{i+1}^{(k)}] \times [y_{j}^{(k)},y_{j+1}^{(k)}]$.
\end{itemize}
\end{remark}

The next corollary is an extension of Corollary~\ref{cor:minkowski} to general quasi-copulas that lie in $\mathcal{S}$. It expresses the Minkowski norm of a quasi-copula $Q \in \mathcal{S}$ with the coefficient $\alpha_Q$ defined in equation \eqref{eq:alpha_general}. It can be seen as a supplement to \cite[Theorem~5.2]{DarOls95}, where the Minkowski norm was expressed with the help of derivatives.
Note, however, that there is a crucial difference between the formula expressed with $\alpha_Q$ and the formula given in \cite[Theorem~5.2]{DarOls95}. Namely, $\alpha_Q$ can be calculated for any quasi-copula, even those that are not contained in $\mathcal{S}$, while the formula in \cite[Theorem~5.2]{DarOls95} assumes apriori that $Q$ is in $\mathcal{S}$.

\begin{corollary}
    For every quasi-copula $Q \in \mathcal{S}$ we have $\|Q\|_M=2\alpha_Q-1$.
\end{corollary}

\begin{proof}
By \cite[Lemma~3.2]{DarOls95} there exist quasi-copulas $A$ and $B$, and real numbers $s,t \ge 0$ such that $Q=sA-tB$ and $\|Q\|_M=s+t$. Evaluating at point $(1,1)$ we get $s-t=1$ and
$$\|Q\|_M=2s-1.$$
Then $V_{Q}(R_{ij}^n)=sV_A(R_{ij}^n)-tV_B(R_{ij}^n)$ and so
$$V_{Q}(R_{ij}^n)^+ \le sV_A(R_{ij}^n)$$
for all $i$, $j$, and $n$.
Summing over all $i$ we obtain
$$\sum_{i=1}^{2^n} V_{Q}(R_{ij}^n)^+ \le s\sum_{i=1}^{2^n} V_A(R_{ij}^n).$$
Since $A$ is a copula we get $\sum_{i=1}^{2^n} V_A(R_{ij}^n) = V_A([\frac{i-1}{2^n},\frac{i}{2^n}] \times \I)=\frac{1}{2^n}$, so $s \ge 2^n \sum_{i=1}^{2^n} V_{Q}(R_{ij}^n)^+$.
Likewise, $s \ge 2^n \sum_{j=1}^{2^n} V_{Q}(R_{ij}^n)^+.$ Hence, $s \ge \alpha_Q$.
This implies $\|Q\|_M \ge 2\alpha_{Q}-1$.

By Theorem~\ref{thm:absolute-convergence} coefficient $\alpha_Q$ is finite. Recall from the proof of implication $(iii) \Longrightarrow (ii)$ of Theorem~\ref{thm:absolute-convergence} that $Q$ was written as $\alpha C-(\alpha-1)D$ for some copulas $C$ and $D$, where $\alpha$ was the limit of some subsequence of $\alpha_n$ defined in \eqref{eq:alpha_n}. Therefore,
$$\|Q\|_M \le 2\alpha-1.$$
Since $\alpha_Q$ is the supremum of $\alpha_n$ over all $n$, we obtain $\alpha \le \alpha_Q$, so $\|Q\|_M \le 2\alpha_Q-1$. Hence, $\|Q\|_M=2\alpha_Q-1$.
\end{proof}

While every copula induces a positive measure on Borell $\sigma$-algebra in $\I^2$, there exist quasi-copulas that do not induce a signed measure on the same $\sigma$-algebra as shown in \cite{FerRodUbe11}.
The characterization of all quasi-copulas that do induce a signed measure is still an open question, see \cite{AriMesDeB20}.
Now, every quasi-copula in $\mathcal{S}$ is a linear combination of two copulas, so it clearly induces a signed measure.
Unfortunately, not every quasi-copula that does induce a signed measure is of this form.
In our next example we construct a quasi-copula which induces a signed measure but cannot be written as a linear combination of two copulas.

\begin{example}\label{exa:final}
For every positive integer $i$ define a discrete quasi-copulas $Q_i$ on an equidistant mesh with $2i+1$ subintervals of $\I$ in such  a way that $Q_i$ has both a positive mass of total size $\frac{(i+1)^2}{2i+1}$ and a negative mass of total size $-\frac{i^2}{2i+1}$ evenly distributed  in a chessboard pattern inside the central diamond-shaped area, and has zero mass outside of this diamond-shape, 
so that $Q_i$ has precisely $(i+1)^2$ squares with mass $\frac{1}{2i+1}$ and exactly $i^2$ squares with mass $-\frac{1}{2i+1}$. 
In particular, the spread of mass for $Q_1$,  $Q_2$, and $Q_3$ is
$$ Q_1\colon \left(\begin{smallmatrix}
0 & +1/3 & 0\\
+1/3 & -1/3 & +1/3\\
0 & +1/3 & 0
\end{smallmatrix}\right), \quad Q_2\colon \left(\begin{smallmatrix}
0 & 0& +1/5 & 0 &0\\
0 & +1/5 & -1/5 & +1/5 &0\\
 +1/5 & -1/5 & +1/5 &-1/5 & +1/5\\
0 & +1/5 & -1/5 & +1/5 &0\\
0 & 0& +1/5 & 0 &0
\end{smallmatrix}\right),
\quad\hbox{and}\quad
$$
$$Q_3\colon \left(\begin{smallmatrix}
0&0 & 0& \alpha_3 & 0 &0&0\\
0&0 & \alpha_3 & -\alpha_3 & \alpha_3 &0&0\\
0& \alpha_3 & -\alpha_3 & \alpha_3 &-\alpha_3 & \alpha_3&0\\
\alpha_3& -\alpha_3 & \alpha_3 & -\alpha_3 &\alpha_3 & -\alpha_3&\alpha_3\\
0 & \alpha_3 & -\alpha_3 & \alpha_3 &-\alpha_3 &\alpha_3 &0\\
0 & 0 & \alpha_3 & -\alpha_3 &\alpha_3  &0 &0\\
0&0 & 0& \alpha_3 & 0 &0&0
\end{smallmatrix}\right);\qquad \alpha_3=1/7.$$
Let a quasi-copula  $\widehat{Q}_i$ be a piecewise bilinear extension of $Q_i$.
Let $a_0=0$ and $a_i=\frac12+\frac14+\ldots+\frac{1}{2^i}$ for all $i\ge 1$, and let  $J_i=[a_{i-1},a_i]\subseteq\I$ for all $i \geq 1$. Note that the union of these intervals is $[0,1)$.
Finally, define  quasi-copula $Q$ as the countable ordinal sum of $\widehat{Q}_i$ with respect to the family of intervals
$\{J_i\}$ (we refer to \cite{FerNelUbe11} and \cite{AriMesDeB20} for the definition and general discussion of ordinal sums of quasi-copulas).
In particular,
$$Q(x,y)=\begin{cases}
a_{i-1} + (a_i-a_{i-1}) \widehat{Q}_i(\tfrac{x-a_{i-1}}{a_i-a_{i-1}},\tfrac{y-a_{i-1}}{a_i-a_{i-1}}); &  (x,y)\in J_i^2,\\
\min\{x,y\} ; & \text{otherwise.}
\end{cases}$$
The total positive mass of $Q$ is $\sum_{i=1}^\infty \frac{1}{2^i} \frac{(i+1)^2}{(2i+1)}$ which is a convergent sum (equal to $\frac{3}{2}+\frac{\sqrt{2}}{4} \mathop{\mathrm{arsinh}} 1$). Similarly, the total negative mass of $Q$ equals $-\sum_{i=1}^\infty \frac{1}{2^i} \frac{i^2}{(2i+1)}$ 
which is also convergent (equal to $ -\frac{1}{2}-\frac{\sqrt{2}}{4} \mathop{\mathrm{arsinh}} 1$). Thus, $Q$ induces a signed measure.

Observe also that the total positive  mass of $Q_i$ in its middle column equals $\frac{i+1}{2i+1}$, so the total positive mass of $Q$ in the corresponding vertical strip, which we denote by  $[x_i,x_{i+1}]\times\I$, equals $\frac{i+1}{2^i(2i+1)}$. Since the width of this vertical strip is $\frac{1}{2^i(2i+1)}$, we see that the total positive mass of $Q$ in the strip  $[x_i,x_{i+1}]\times\I$ divided by strip width equals $i+1$. Hence, the
supremum in $(iii')$ of Remark~\ref{rem:main} is infinite.
By Theorem~\ref{thm:absolute-convergence} and Remark~\ref{rem:main}, $Q$ cannot be written as a linear combination of two copulas.
\end{example}

\section{Conclusions}
In this paper we investigated when a bivariate quasi-copula can be expressed as a (finite or infinite) linear combination of bivariate copulas.
We gave a new direct proof that every discrete quasi-copula is a linear combination of two discrete copulas and showed that every general bivariate quasi-copula can be written as a uniformly convergent infinite sum of multiples of copulas.
We also characterized when a general quasi-copula can be expressed as a linear combination of two copulas by giving an equivalent condition involving the mass distribution of the quasi-copula.

One of our main methods was  mass domination. We illustrated how the mass distribution of a discrete $2$-increasing function can be dominated from above by a mass distribution of a multiple of a discrete copula.
We determined the optimal constant for such domination and the lower and upper bound for the set of corresponding discrete copulas.
This allowed us to express the Minkowski norm of a discrete quasi-copula in terms of its mass distribution.

Since infinite series (of multiples) of copulas seem to be largely unexplored, we hope this paper will encourage future investigations in this direction.
Although such series take us outside the framework of quasi-copulas in general, it has been the case in several of our recent results on copulas that the proofs required to go  beyond quasi-copulas into classes of more general functions.
For example, in \cite{OmlSto20} general, even noncontinuous, real-valued functions were used to construct copulas between two quasi-copulas.
The proof of exactness of bounds for copulas with fixed value at one point in \cite{KleKokOmlSamSto22} required the introduction of a new class of functions, called $F$-copulas, which generalize copulas. And in \cite{Sto23} envelopes of distribution functions were represented with semi-copulas constructed from quasi-copulas.
In order to better understand (quasi-) copulas it therefore seems to be beneficial to study larger classes of function.

\section*{Acknowledgments}

The authors acknowledge financial support from the ARIS (Slovenian Research and Innovation Agency, research core funding No. P1-0288, P1-0285, and P1-0222, and research project N1-0210).

\bibliographystyle{amsplain}
\bibliography{LCC_arXiv}

\end{document}